\documentclass[reqno, 12pt]{amsart}

\usepackage{array}
\usepackage{amsmath}
\usepackage{amsfonts}
\usepackage{amssymb}
\usepackage{enumerate}
\usepackage{amsthm}
\usepackage{amsmath, amscd}
\usepackage{bm}
\usepackage{xy}
\xyoption{all}
\usepackage{color}
\usepackage{marvosym}
\usepackage{amsbsy}
\usepackage{hyperref}
\usepackage{tikz}
\usetikzlibrary{matrix,arrows,backgrounds}%, decorations, pathmorphing, positioning, fit}
	
\newtheorem{introtheorem}{Theorem}

\newtheorem{theorem}{Theorem}[subsection]

\newtheorem{lemma}[theorem]{Lemma}
\newtheorem{proposition}[theorem]{Proposition}
\newtheorem{corollary}[theorem]{Corollary}

\theoremstyle{definition}
\newtheorem{definition}[theorem]{Definition}
\newtheorem{remark}[theorem]{Remark}

\newtheorem{example}[theorem]{Example}

\newcommand{\op}[1]{\operatorname{#1}}

\newcommand{\leftexp}[2]{{\vphantom{#2}}^{#1}{#2}}

\newcommand{\dbcoh}[1]{\operatorname{D}^{\operatorname{b}}(\operatorname{coh }#1)}

\newcommand{\weezer}{\leftexp{=}{\kern-0.23em\operatorname{W}}^{\kern-0.21em =}}

\def\Z{\op{\mathbb{Z}}}
\def\C{\op{\mathbb{C}}}

\title[Wall crossing for derived categories]{Wall crossing for derived categories of moduli spaces of sheaves on rational surfaces}

\author{Matthew Robert Ballard}
\address{
  \begin{tabular}{l}
   Matthew Robert Ballard  \\ 
   \hspace{.1in} University of South Carolina, Columbia, SC, USA \\
   \hspace{.1in} Email: {\bf ballard@math.sc.edu} \\
  \end{tabular}
}

\numberwithin{equation}{section}
\begin{document}
\renewcommand{\labelenumi}{\emph{\alph{enumi})}}

\begin{abstract}
 We remove the global quotient presentation input in the theory of windows in derived categories of smooth Artin stacks of finite type. As an application, we use existing results on flipping of strata for wall-crossing of Gieseker semi-stable torsion-free sheaves of rank two on rational surfaces to produce semi-orthogonal decompositions relating the different moduli stacks. The complementary pieces of these semi-orthogonal decompositions are derived categories of products of Hilbert schemes of points on the surface.  
\end{abstract}

\maketitle

\section{Introduction} \label{section: Introduction}

A central question in the theory of derived categories is the following: given a smooth, projective variety $X$, how does one find interesting semi-orthogonal decompositions of its derived category, $\dbcoh{X}$? Historically, two different parts of algebraic geometry have fed this question: birational geometry and moduli theory. \cite{Orl92,BO95,OrlK3,Bridgeland-flops,KawD-K,Kaw06,Kuz09b} provides a non-exhaustive highlight reel for this approach. 

This paper focuses on the intersection of birational geometry and moduli theory. Namely, given some moduli problem equipped with a notion of stability, variation of the stability condition often leads to birational moduli spaces. As such, it is natural to compare the derived categories in this situation. Let us consider the well-understood situation of torsion-free rank two semi-stable sheaves on rational surfaces \cite{EG,FQ,MW}. The flipping of unstable strata under change of polarization was investigated to understand the change in the Donaldson invariants. It provides the input for the following result, which can be viewed as a categorification of the wall-crossing formula for Donaldson invariants.

\begin{introtheorem}[Corollary~\ref{corollary: SOD for whole wall crossing}] \label{theorem: intro}
 Let $S$ be a smooth rational surface over $\C$ with $L_-$ and $L_+$ ample lines bundles on $S$ separated by a single wall defined by unique divisor $\xi$ satisfying
 \begin{gather*}
  L_- \cdot \xi < 0 < L_+ \cdot \xi \\
  0 \leq \omega_S^{-1} \cdot \xi.
 \end{gather*}
 Let $\mathcal M_{L_{\pm}}(c_1,c_2)$ be the $\mathbb{G}_m$-rigidified moduli stack of Gieseker $L_{\pm}$-semi-stable torsion-free sheaves of rank $2$ with first Chern class $c_1$ and second Chern class $c_2$.
 
 There is a semi-orthogonal decomposition 
 \begin{gather*}
  \dbcoh{ \mathcal M_{L_+}(c_1,c_2) } = \left\langle \underbrace{\dbcoh{H^{l_{\xi}}}, \ldots,  \dbcoh{H^{l_{\xi}}}}_{\mu_{\xi}}, \ldots  \right. \\
  \left. \underbrace{\dbcoh{H^0}, \ldots, \dbcoh{H^0}}_{\mu_{\xi}}, \dbcoh{ \mathcal M_{L_-}(c_1,c_2)} \right\rangle 
 \end{gather*}
 where 
 \begin{align*}
  l_{\xi} & := (4c_2 - c_1^2 + \xi^2)/4 \\
  H^l & := \op{Hilb}^l(S) \times \op{Hilb}^{l_{\xi} - l}(S) \\
  \mu_{\xi} & := \omega_S^{-1} \cdot \xi
 \end{align*}
 with the convention that $\op{Hilb}^0(S) := \op{Spec} \C$.
\end{introtheorem}

While Theorem~\ref{theorem: intro} is interesting in its own right, the method might be moreso. Indeed, Theorem~\ref{theorem: intro} represents one of multiple possible applications, including to moduli spaces of Bridgeland semi-stable objects on rational surfaces, \cite{ABCH}.

Theorem~\ref{theorem: intro} follows from the general technology that goes under the heading of windows in derived categories. Windows provide a framework for addressing the central question put forth above; they are a machine for manufacturing interesting semi-orthogonal decompositions of $\dbcoh{X}$. They have a rich history with contributions by many mathematicians and physicists \cite{KawFF,VdB,Orl09,HHP,Seg2,HW,Shipman,HL12,BFKGIT,DS}. However, windows have not yet achieved their final form. Previous work dealt with an Artin stack $\mathcal X$ plus a choice of global quotient presentation $\mathcal X = [X/G]$. Locating an appropriate quotient presentation for a given $\mathcal X$ is not convenient in applications, in particular the above, so one would like a definition of a window more intrinsic to $\mathcal X$. Section~\ref{section: SODs and BB} provides such a definition using an appropriate type of groupoid in Bia\l{}ynicki-Birula strata and extends the prior results on semi-orthogonal decompositions \cite{BFKGIT} to this setting, see Theorem~\ref{theorem: elementary wall crossing}. A similar extension should appear in \cite{HL14}. 

\vspace{2.5mm}
\noindent \textbf{Acknowledgments:}
The author has benefited immensely from conversations and correspondence with Arend Bayer, Yujiro Kawamata, Colin Diemer, Gabriel Kerr, Ludmil Katzarkov, and Maxim Kontsevich, and would like to thank them all for their time and insight. The author would especially like to thank David Favero for a careful reading of and useful suggestions on a draft of this manuscript. Finally, the author would also like to thank his parents for fostering a hospitable work environment. 

The author was supported by a Simons Collaboration Grant.
\vspace{2.5mm}

\section{Semi-orthogonal decompositons and BB-strata} \label{section: SODs and BB}

For the whole of this section, $k$ will denote an algebraically-closed field. The term, variety, means a separated, reduced scheme of finite-type over $k$. All points of a variety are closed points unless otherwise explicitly stated.

In this section, we extend the results of \cite{BFKGIT} by removing the global quotient presentation from the input data. We try to keep this section as self-contained as possible.

\subsection{Truncations of sheaves on BB strata}

We begin with the following definition. Let $X$ be a smooth quasi-projective variety equipped with a $\mathbb{G}_m$-action. Let $X^{\mathbb{G}_m}$ denote the fixed subscheme of the action and let $X_0$ be a choice of a connected component of the fixed locus. We recall the following well-known result of Bia\l{}ynicki-Birula. 

\begin{theorem} \label{theorem: BB}
 The fixed locus $X^{\mathbb{G}_m}$ is smooth and is a closed subvariety of $X$. Let $X_0$ be a connected component of $X^{\mathbb{G}_m}$. There exists a unique smooth and locally-closed $\mathbb{G}_m$-invariant subvariety $X^+_0$ of $X$ and a unique morphism $\pi: X^+_0 \to X_0$ such that 
 \begin{enumerate}
  \item $X_0$ is a closed subvariety of $X^+_0$.
  \item The morphism $\pi: X_0^+ \to X _0$ is an equivariantly locally-trivial fibration of affine spaces over $X_0$.
  \item For a point $x \in X_0$, there is an equality 
  \begin{displaymath}
   T_x X_0^+ = \left(T_x X\right)^{\geq 0}
  \end{displaymath}
  where the right hand side is the subspace of non-negative weights of the geometric tangent space.
 \end{enumerate}
\end{theorem}

\begin{proof}
 This is \cite[Theorem 2.1 and Theorem 4.1]{BB}
\end{proof}

\begin{remark}
 Note that the weights on the affine fibers are all positive and the set of closed points of a $X^+_0$ from Theorem~\ref{theorem: BB} is 
 \begin{displaymath}
  \{ x \in X \mid \lim_{\alpha \to 0} \sigma(\alpha,x) \in X_0 \}
 \end{displaymath}
 so one may think of $X^+_0$ as the set of points that flow into $X_0$ as $\alpha \to 0$. 
\end{remark}

\begin{definition} \label{definition: BB strata}
 Let $X$ be a smooth quasi-projective variety equipped with a $\mathbb{G}_m$ action. Let $X_0$ be a choice of a connected component of $X^{\mathbb{G}_m}$. The \textbf{BB stratum} associated with $X_0$ is $X^+_0$ appearing in Theorem~\ref{theorem: BB}. 
 
 If $X = X^+_0$, then we shall say also say that $X$ is a BB stratum. In particular, we require that $X^{\mathbb{G}_m}=X_0$ is connected in this situation. 
\end{definition}

\begin{remark}
 Note that $X$ being a BB stratum is equivalent to an extension of the action map $\sigma: \mathbb{G}_m \times X \to X$ to morphism $\mathbb{A}^1 \times X \to X$. 
\end{remark}

Let $X$ be a BB stratum. Let us consider the local situation first. So $X = \op{Spec} R[x_1,\ldots,x_n]$ and we have a coaction map, also denoted by $\sigma$,
\begin{displaymath}
 \sigma: R[x_1,\ldots,x_n] \to R[x_1,\ldots,x_n,u,u^{-1}]
\end{displaymath}
where the weights the $x_i$ are positive. The fixed locus is then $X_0 = \op{Spec} R$. Let $M$ be a $\mathbb{G}_m$ equivariant module over $R$, i.e. a quasi-coherent $\mathbb{G}_m$-equivariant sheaf on $X_0$. Then, we have a map
\begin{displaymath}
 \Delta: M \to M[u,u^{-1}]
\end{displaymath}
corresponding to the equivariant structure. One sets
\begin{displaymath}
 M_i := \{ m \in M \mid \Delta(m) = m \otimes u^i \}.
\end{displaymath}
If we try to globalize this construction, then two different $M_i$'s are identified under an automorphism of $M$, which has degree zero with respect to $\mathbb{G}_m$. Thus, for any quasi-coherent $\mathbb{G}_m$-equivariant sheaf, this gives a quasi-coherent $\mathbb{G}_m$-quasi-coherent sheaf on $X_0$, $\mathcal E_i$. 

\begin{lemma} \label{lemma: decomposition on the fixed locus} 
 Let $X$ be a BB stratum and let $\mathcal E$ be a coherent $\mathbb{G}_m$-equivariant sheaf on the fixed locus $X_0$. Then there is a functorial decomposition 
 \begin{displaymath}
  \mathcal E \cong \bigoplus_{i \in \Z} \mathcal E_i. 
 \end{displaymath}
 In particular, for each $i \in \Z$, the functor 
 \begin{displaymath}
  \mathcal E \mapsto \mathcal E_i
 \end{displaymath}
 is exact. 
\end{lemma}

\begin{proof}
 This is standard and a proof is suppressed.
\end{proof}

\begin{corollary} \label{corollary: decomposition on the fixed locus derived version}
 Let $X$ be a BB stratum and let $\mathcal E$ be a bounded complex of coherent $\mathbb{G}_m$-equivariant sheaves on the fixed locus $X_0$. Then there is a functorial decomposition of the complex
 \begin{displaymath}
  \mathcal E \cong \bigoplus_{i \in \Z} \mathcal E_i. 
 \end{displaymath}
 This descends to the derived category, $\dbcoh{[X_0/\mathbb{G}_m]}$.
\end{corollary}

\begin{proof}
 This follows immediately from Lemma~\ref{lemma: decomposition on the fixed locus}. 
\end{proof}

\begin{definition} \label{definition: weight decomposition}
 For a subset $I \subseteq \Z$, we say that a complex $\mathcal E$ from $\dbcoh{[X_0/\mathbb{G}_m]}$ has \textbf{weights concentrated in $I$} if $(\mathcal H^p (\mathcal E))_i = 0$ for all $p \in \Z$ and $i \not \in I$.
\end{definition}

Now, we turn our attention to $\mathbb{G}_m$-equivariant sheaves on $X$ itself. Let $j: X_0 \to X$ be the inclusion. 

\begin{definition}
 Let $\mathcal E$ be an object $\dbcoh{[X/\mathbb{G}_m]}$ and let $I \subseteq \Z$. We say that $\mathcal E$ has \textbf{weights concentrated in $I$} if $\mathbf{L}j^* \mathcal E$ has weights concentrated in $I$. 
\end{definition}

Next, we want to give a procedure to truncate the weights. We first again go back to the local case with $X = \op{Spec} R[x_1,\ldots,x_n]$ and the $x_i$'s having positive weight. Let $M$ be a $\mathbb{G}_m$-equivariant module over $R[x_1,\ldots,x_n]$. We can still consider 
\begin{displaymath}
 M_a : = \{ m \in M \mid \sigma(m) = m \otimes u^a \}.
\end{displaymath}
However, this is no longer a submodule of $M$ as multiplication by $x_i$ will raise the weight. But,
\begin{displaymath}
 M_{\geq a} := \bigoplus_{j \geq a} M_j
\end{displaymath}
is a submodule of $M$ and inherits a natural $\mathbb{G}_m$-equivariant structure. The assignment $M \mapsto M_{\geq a}$ is functorial with respect to $\mathbb{G}_m$-equivariant morphisms so gives a exact functor
\begin{displaymath}
 \tau_{\geq a} : \op{coh} [X/\mathbb{G}_m] \to \op{coh} [X/\mathbb{G}_m]
\end{displaymath}
which, of course, descends to the derived category
\begin{displaymath}
 \tau_{\geq a}: \dbcoh{[X/\mathbb{G}_m]} \to \dbcoh{[X/\mathbb{G}_m]}. 
\end{displaymath}

Now let us consider the global situation. Since we have a $\mathbb{G}_m$-invariant cover of $X$ of the form $\op{Spec} R[x_1,\ldots,x_n]$ with $x_i$ having positive weights and $R$ have zero weights, we can glue this construction to get
\begin{align*}
 \tau_{\geq a}: \dbcoh{[X/\mathbb{G}_m]} & \to \dbcoh{[X/\mathbb{G}_m]} \\
 \mathcal E & \to \mathcal E_{\geq a}.
\end{align*}

\begin{definition}
 Let $a \in \Z$. For a bounded complex of coherent $\mathbb{G}_m$-equivariant sheaves on $X$, $\mathcal E$, one calls $\mathcal E_{\geq a}$ a \textbf{weight truncation} of $\mathcal E$. 
\end{definition}

\begin{lemma} \label{lemma: truncation of weights actually truncates weights}
 If $\mathcal E$ has weights concentrated in $I$, then $\tau_{\geq a} \mathcal E$ has weights concentrated in $I \cap [a, \infty)$. Moreover, if $\mathcal E$ has weights concentrated in $[a, \infty)$, then there is a natural quasi-isomorphism $\mathcal E \cong \mathcal E_{\geq a}$. 
\end{lemma}

\begin{proof}
 We can check this computation locally and assume that $\mathcal E$ is a bounded complex of locally-free sheaves. Then, one sees that there is a natural isomorphism
 \begin{displaymath}
  (\mathcal E_{\geq a})|_{X_0} \cong (\mathcal E|_{X_0})_{\geq a}.
 \end{displaymath}
 Looking at the left-hand side we see that the complex has weights in $[a, \infty) \cap I$. 
\end{proof}

Now we level up and consider an appropriate type of groupoid in BB strata. Let $X^1 \overset{s}{\underset{t}\rightrightarrows} X^0$ be a groupoid scheme with $s$, $t$ smooth and $X^1,X^0$ smooth and quasi-projective. In general, we shall suppress the additional data packaged in a groupoid scheme, including in the next statement. Assume we have a commutative diagram
\begin{center}
 \begin{tikzpicture}[description/.style={fill=white,inner sep=2pt}]
 \matrix (m) [matrix of math nodes, row sep=2em, column sep=1.5em, text height=1.5ex, text depth=0.25ex]
 {  \mathbb{G}_m \times X^0 & X^0 \\
    X^1 & X^0. \\ };
 \path[->,font=\scriptsize]
  ([yshift=2pt]m-1-1.east) edge node [above] {$\pi$} ([yshift=2pt]m-1-2.west)
  ([yshift=-2pt]m-1-1.east) edge node [below] {$\sigma$} ([yshift=-2pt]m-1-2.west)
  (m-1-1) edge node [left] {$l$} (m-2-1)
  ([yshift=2pt]m-2-1.east) edge node [above] {$s$} ([yshift=2pt]m-2-2.west)
  ([yshift=-2pt]m-2-1.east) edge node [below] {$t$} ([yshift=-2pt]m-2-2.west)
  (m-1-2) edge node [right] {$=$} (m-2-2)
 ;
 \end{tikzpicture}
\end{center}
of groupoid schemes with $l$ a closed embedding and $\sigma$ an action. Then, we can define a morphism 
\begin{align*}
 A: \mathbb{G}_m \times X^1 & \to X^1 \\
 (\alpha, x_1) & \mapsto l(\alpha,t(x_1)) \cdot x_1 \cdot l(\alpha^{-1},\sigma(\alpha,s(x_1)))
\end{align*}
where the central dot is notation for the multiplication $m: X_1 \times_{s,X^0,t} X^1 \to X^1$. 

\begin{lemma} \label{lemma: adjoint action of Gm}
 The morphism $A$ defines an action of $\mathbb{G}_m$ on $X^1$ making both $s$ and $t$ equivariant. 
\end{lemma}

\begin{proof}
 This is straightforward to verify so the details are suppressed. 
\end{proof}

\begin{definition} \label{definition: BB groupoid}
 If we have a commutative diagram of groupoids as above, the action $A$ is the called the \textbf{adjoint} of $\mathbb{G}_m$ on $X^1$. We say that a groupoid scheme $X^{\bullet}$ is a \textbf{stacky BB stratum} if $A$ extends to a morphism $\mathbb{A}^1 \times X^1 \to X^1$ and $(X^0)^{\mathbb{G}_m}$ is connected. Similarly, we say the associated stack $[X^0/X^1]$ is a stacky BB stratum if $X^{\bullet}$ is such. 
 
 We can pass to the fixed loci of the $\mathbb{G}_m$ actions on $X^{\bullet}$ to get another groupoid scheme $(X^1)^{\mathbb{G}_m} \rightrightarrows (X^0)^{\mathbb{G}_m}$ which we call the \textbf{fixed substack} and denoted by $[X^0/X^1]_0$. Taking the limit as $\alpha \to 0$ in $\mathbb{G}_m$ we get an induced projection denoted by $\pi: [X^0/X^1] \to [X^0/X^1]_0$. 
\end{definition}

\begin{remark}
 Note that if $A$ extends to a morphism $\mathbb{A}^1 \times X^1 \to X^1$ then $X^1$ and $X^0$ are unions of BB strata. For simplicity of exposition, we require the connectedness of $(X^0)^{\mathbb{G}_m}$. The disconnected case can be handled using the same arguments with minor modifications. 
\end{remark}

\begin{lemma} \label{lemma: truncation well-defined operation on stacky BB-strata}
 Let $X^{\bullet}$ be a stacky BB stratum and let $\tau_{\geq l}$ be the trunction functor on $\dbcoh{[X^0/\mathbb{G}_m]}$ as previously defined. Then, $\tau_{\geq l}$ descends to an endofunctor
 \begin{displaymath}
  \tau_{\geq l} : \dbcoh{ X^{\bullet} } \to \dbcoh{ X^{\bullet} }. 
 \end{displaymath}
 Furthermore, the weight decomposition on $\dbcoh{[X^0_0/\mathbb{G}_m]}$ descends to $\dbcoh{[X^0_0/X^1_0]}$. 
\end{lemma}

\begin{proof}
 Recall that a quasi-coherent sheaf on $X^{\bullet}$ is a pair $(\mathcal E, \theta)$ with $\mathcal E \in \op{Qcoh}(X^0)$ and isomorphism $\theta: s^* \mathcal E \overset{\sim}{\to} t^* \mathcal E$ satisfying an appropriate cocycle condition and identity condition. Note that any sheaf on $X^{\bullet}$ carries a $\mathbb{G}_m$-equivariant structure by pulling back along $l$, so truncation is well-defined on $\mathcal E$. For any coherent sheaf $\mathcal E$ on $X^{\bullet}$, we will show that
 \begin{displaymath}
  \theta(t^* (\tau_{\geq l} \mathcal E)) \subset s^*(\tau_{\geq l} \mathcal E). 
 \end{displaymath}
 This suffices to show that $\tau_{\geq l}$ descends as $\theta^{-1} = i^*\theta$ where $i: X^1 \to X^1$ is the inverse over $X^0$.
 
 This question is local so we may assume that $X^1 = \op{Spec} S$ and $X^0 = \op{Spec} R$. Let $M$ be the module corresponding to $\mathcal E$ so that $\theta : M \otimes_{R,t} S \to M \otimes_{R,s} S$ is an isomorphism. As we will need to use it, we now recall the cocycle condition in this local situation. We have a diagram
 \begin{center}
 \begin{tikzpicture}[description/.style={fill=white,inner sep=2pt}]
 \matrix (m) [matrix of math nodes, row sep=2em, column sep=1.5em, text height=1.5ex, text depth=0.25ex]
 { M \otimes_{R,t} S \otimes_{s,R,t} S & M \otimes_{R,t} S \otimes_{s,R,t} S \\
   S \otimes_{s,R,t} S \otimes_{s,R} M & S \otimes_{s,R,t} S \otimes_{t,R} M. \\ };
 \path[->,font=\scriptsize]
  (m-1-1) edge node [above] {$\theta \otimes_R t$} (m-1-2)
  (m-1-2) edge node [right] {$\sim$} (m-2-2)
  (m-2-2) edge node [below] {$s \otimes_R \theta$} (m-2-1)
  (m-2-1) edge node [left] {$\sim$} (m-1-1)
 ;
 \end{tikzpicture}
 \end{center}
 which gives us an isomorphism
 \begin{displaymath}
  M \otimes_{R,t} S \otimes_{s,R,t} S \overset{ (s \otimes_R \theta) \circ (\theta \otimes_R t) }{\to} M \otimes_{R,t} S \otimes_{s,R,t} S.
 \end{displaymath}
 We can also tensor $\theta$ over $S$ with $m : S \to S \otimes_{s,R,t} S$ to get another isomorphism. The cocycle condition is 
 \begin{displaymath}
  (s \otimes_R \theta) \circ (\theta \otimes_R t) = \theta \otimes_S m.
 \end{displaymath}

 Now, we wish to check that $\theta(M_i \otimes_R S) \subset M_{\geq i} \otimes_R S$. We do this as follows. First, since $A: S \to S[u,u^{-1}]$ has image in $S[u]$, we must have that the image of
 \begin{displaymath}
  (l \otimes 1 \otimes l) \circ (m \otimes 1) \circ m := \tilde{A}: S \to R[u_1,u_1^{-1}] \otimes_{\pi, R, t} S \otimes_{s, R, \sigma} R[u_2,u_2^{-1}]
 \end{displaymath}
 lies in the $S$-subalgebra generated by $u_1^i u_2^j$ with $i - j \geq 0$. Since $\theta$ satisfies the cocyle condition, we can factor
 \begin{displaymath}
  \theta \otimes_S \tilde{A} = (1 \otimes_R \theta \otimes_R 1) \circ (\overline{\theta} \otimes_R 1 \otimes_R 1) \circ (1 \otimes_R \theta \otimes_R 1) \circ (1 \otimes_R 1 \otimes_R \overline{\theta}).
 \end{displaymath}
 Applying this to $m_j \in M_j$, we have 
 \begin{displaymath}
  (\theta \otimes_S \tilde{A})(m_j) = \sum_i \theta(\theta(m_j)_i)) u_1^iu_2^j
 \end{displaymath}
 where
 \begin{displaymath}
  \theta(m_j) = \sum_i \theta(m_j)_i u_1^i.
 \end{displaymath}
 As $i-j \geq 0$, we see that the $\theta(m_j) \in M_{\geq j} \otimes_{R,s} S$ and $\theta$ preserves the truncation.
 
 For the final statement, repeating the previous argument and assuming $\mathbb{G}_m$ acts trivially on $R$ and $S$ shows that $\theta$ preserves the whole splitting via weights. 
\end{proof}

\begin{example} \label{example: BFK example}
 Let $G$ be a linear algebraic group acting on a variety $X$. Assume we have a one-parameter subgroup, $\lambda: \mathbb{G}_m \to G$, and choice of connected component of the fixed locus, $Z_{\lambda}^0$. Then, we have the BB stratum $Z_{\lambda}$ and its orbit $S_{\lambda} := G \cdot Z_{\lambda}$. Let 
 \begin{displaymath}
  P(\lambda) := \{ g \in G \mid \lim_{\alpha \to 0} \lambda(\alpha) g \lambda(\alpha)^{-1} \text{ exists }\}.
 \end{displaymath}
 There is an induced action of $P(\lambda)$ on $Z_{\lambda}$. In general, $G \overset{P(\lambda)}{\times} Z_{\lambda}$ is a resolution of singularities of $S_{\lambda}$. If we assume this map is an isomorphism, then we get what is called \textbf{an elementary stratum} in the language of \cite{BFKGIT}. Since
 \begin{displaymath}
  [G \overset{P(\lambda)}{\times} Z_{\lambda} / G ] \cong [Z_{\lambda} / P(\lambda) ]
 \end{displaymath}
 and the groupoid,
 \begin{displaymath}
  P(\lambda) \times Z_{\lambda} \overset{\pi}{\underset{\sigma}\rightrightarrows} Z_{\lambda}
 \end{displaymath}
 is a stacky BB stratum, we see that $[S_{\lambda}/G]$ is a stacky BB stratum. The simplest case is $G = \mathbb{G}_m$. 
\end{example}

\subsection{Removing stacky BB strata and comparing derived categories}

Let $\mathcal X$ be a smooth Artin stack of finite type over $k$.

\begin{definition} \label{definition: weights along stacky fixed locus}
 Let $i: \mathcal Z \to \mathcal X$ be a smooth closed substack that is also a stacky BB stratum and let $l: \mathcal Z_0 \to \mathcal X$ be the closed immersion of the fixed substack. Let $\mathcal E$ be a bounded complex of coherent sheaves $\mathcal X$ and let $I \subseteq \Z$. We say that $\mathcal E$ has \textbf{weights concentrated in $I$} if $\mathbf{L}l^* \mathcal E \in \dbcoh{\mathcal Z_0}$ has weights concentrated in $I$.
\end{definition}

\begin{definition} \label{definition: windows}
 Let $\mathcal Z$ be a stacky BB stratum in $\mathcal X$ and let $I \subseteq \Z$. The \textbf{$I$-window} associated to $\mathcal Z$ is the full subcategory whose objects have weights concentrated in $I$. We denote this subcategory by $\weezer(I,\mathcal Z_0)$. 
\end{definition}

\begin{lemma} \label{lemma: fully-faithful}
 Assume that the weights of the conormal sheaf of $\mathcal Z$ in $\mathcal X$ are all strictly negative and let $t_{\mathcal Z}$ be the weight of the relative canonical sheaf $\omega_{\mathcal Z \mid \mathcal X}$. Set $\mathcal U := \mathcal X \setminus \mathcal Z$ and $j: \mathcal U \to \mathcal X$ be the inclusion. Then, the functor
 \begin{displaymath}
  j^*: \weezer(I,\mathcal Z_0) \to \dbcoh{\mathcal U}
 \end{displaymath}
 is fully-faithful whenever $I$ is contained in a closed interval of length $< -t_{\mathcal Z}$. 
\end{lemma}

\begin{proof}
 The argument here is essentially due to Teleman \cite[Section 2]{Tel} which the author first learned of in \cite{HL12}. It amounts to descending through a few spectral sequences. For the convenience of the reader and to keep the paper self-contained, we recapitulate it in some detail. 
 
 For any two objects $\mathcal E,\mathcal F$ of $\dbcoh{\mathcal X}$ there is an exact triangle of graded vector spaces 
 \begin{displaymath}
  \op{Hom}_{\mathcal X, \mathcal Z}(\mathcal E, \mathcal F) \to \op{Hom}_{\mathcal X}(\mathcal E, \mathcal F) \overset{j^*}{\to} \op{Hom}_{\mathcal U}(j^* \mathcal E, j^* \mathcal F) 
 \end{displaymath}
 coming from applying the exact triangle of derived functors
 \begin{equation} \label{equation: local cohomology global sections}
  \mathbf{R}\Gamma_{\mathcal Z}(\mathcal X,-) \to \mathbf{R}\Gamma(\mathcal X,-) \overset{j^*}{\to} \mathbf{R}\Gamma(\mathcal U,-) 
 \end{equation}
 to $\mathbf{R}\mathcal Hom_{\mathcal X}(\mathcal E, \mathcal F)$. Therefore, a necessary and sufficient condition for $j^*$'s fully-faithfulness is the vanishing of $\op{Hom}_{\mathcal X, \mathcal Z}(\mathcal E, \mathcal F)$. 
 
 The exact triangle in Equation~\eqref{equation: local cohomology global sections} comes from applying $\mathbf{R}\Gamma(\mathcal X,-)$ to the following exact triangle 
 \begin{equation*}
  \mathcal H_{\mathcal Z} \to \op{Id} \to \mathbf{R}j_*j^*
 \end{equation*}
 of functors. Here $\mathcal H_{\mathcal Z}$ is the derived sheafy local cohomology functor. The complex $\mathcal H_{\mathcal Z}(\mathcal G)$ is not scheme-theoretically supported on $\mathcal Z$ but there is a bounded above filtration by powers of the ideal sheaf $\mathcal I_{\mathcal Z}$. The associated graded sheaves are scheme-theoretically supported on $\mathcal Z$. Furthermore, since $\mathcal Z$ is smooth, the $s$-th associated graded piece is isomorphic to $\mathbf{L}i^*\mathcal G \otimes \op{Sym}^s(\mathcal T_{\mathcal Z \mid \mathcal X}) \otimes \omega^{-1}_{\mathcal Z \mid \mathcal X}$. 
 
 We may now take global sections on $\mathcal Z$ which we may factor through two pushforwards: one by $\pi: \mathcal Z \to \mathcal Z_0$ and one by the rigidification map $r: \mathcal Z_0 \to \mathcal Z^{\mathbb{G}_m}_0$ \cite[Theorem 5.15]{ACV}. The pushforward $r_*: \dbcoh{\mathcal Z_0} \to \dbcoh{\mathcal Z_0^{\mathbb{G}_m}}$ projects onto the weight $0$ component of the weight decomposition. Thus, to establish vanishing, it suffices to show that there is no weight $0$ component in the splitting of 
 \begin{displaymath}
  \pi_* \left( \mathbf{L}i^*\mathcal G \otimes \op{Sym}^s(\mathcal T_{\mathcal Z \mid \mathcal X}) \otimes \omega^{-1}_{\mathcal Z \mid \mathcal X} \right).
 \end{displaymath}
 Since we can resolve anything using pullbacks from $\mathcal Z_0$, we can use the projection formula to get 
 \begin{displaymath}
  \pi_* \left( \mathbf{L}i^*\mathcal G \otimes \op{Sym}^s(\mathcal T_{\mathcal Z \mid \mathcal X}) \otimes \omega^{-1}_{\mathcal Z \mid \mathcal X} \right) \cong \mathbf{L}l^* \mathcal G \otimes \op{Sym}^s(\mathcal T_{\mathcal Z \mid \mathcal X})|_{\mathcal Z_0} \otimes \omega^{-1}_{\mathcal Z \mid \mathcal X}|_{\mathcal Z_0} \otimes \pi_* \mathcal O_{\mathcal Z}.
 \end{displaymath}
 One notices that the terms 
 \begin{displaymath}
  \op{Sym}^s(\mathcal T_{\mathcal Z \mid \mathcal X})|_{\mathcal Z_0} \otimes \omega^{-1}_{\mathcal Z \mid \mathcal X}|_{\mathcal Z_0} \otimes \pi_* \mathcal O_{\mathcal Z}
 \end{displaymath}
 have weights $\geq -t_{\mathcal Z}$. When the weights of $\mathcal G$ are concentrated in $(t_{\mathcal Z}, \infty)$, we get a trivial weight zero component and the desired vanishing. Now, setting $\mathcal G = \mathbf{R}\mathcal Hom_{\mathcal X}(\mathcal E, \mathcal F)$ for $\mathcal E, \mathcal F \in \weezer(I,\mathcal Z_0)$, we see that our assumption on $I$ exactly implies this.
\end{proof}

\begin{lemma} \label{lemma: essentially surjective}
 Assume that the weights of the conormal sheaf of $\mathcal Z$ in $\mathcal X$ are all strictly negative and let $t_{\mathcal Z}$ be the weight of the relative canonical sheaf $\omega_{\mathcal Z \mid \mathcal X}$. Set $\mathcal U := \mathcal X \setminus \mathcal Z$ and $j: \mathcal U \to \mathcal X$ be the inclusion. Then, the functor
 \begin{displaymath}
  j^*: \weezer(I,\mathcal Z_0) \to \dbcoh{\mathcal U}
 \end{displaymath}
 is essentially surjective whenever $I$ contains a closed interval of length $\geq -t_{\mathcal Z} - 1$. 
\end{lemma}

\begin{proof}
 We may assume that $\mathcal X \not = \mathcal Z$ and we may reduce to $I$ being an interval of length $-t_{\mathcal Z} -1$. To demonstrate essential surjectivity it suffices to iteratively reduce the weights by forming exact triangles 
 \begin{displaymath}
  \mathcal E' \to \mathcal E \to \mathcal T
 \end{displaymath}
 where $\mathcal T$ is set-theoretically supported on $\mathcal Z$ and the weights of $\mathcal E'$ are concentrated in a strictly smaller interval than $\mathcal E$. There is an exact triangle 
 \begin{displaymath}
  \tau_{\geq b-1} \mathbf{L}i^* \mathcal E \to \mathbf{L}i^*\mathcal E \to \mathcal W
 \end{displaymath}
 where weights of $\mathcal E$ along $\mathcal Z$ are concentrated in $[a,b]$ with $b-a \geq -t_{\mathcal Z}$. By Lemma~\ref{lemma: truncation of weights actually truncates weights}, the weights of $\tau_{\geq b-1} \mathbf{L}i^* \mathcal E$ are concentrated in $[a,b-1]$. Then, the only weight of $\mathcal W$ is $b$. Consider $i_*\mathcal W$ as an object of $\dbcoh{\mathcal X}$. Let us check that the weights of $i_*\mathcal W$ lie in $[b+t_{\mathcal Z},b]$. We must compute 
 \begin{displaymath}
  \mathbf{L}i^* i_* \mathcal W.
 \end{displaymath}
 This has cohomology sheaves isomorphic to $\mathcal W \otimes \bigwedge^*(\Omega_{\mathcal Z \mid \mathcal X})$ which by assumption has weights in $[b+t_{\mathcal Z},b]$. Note that the only contribution to weight $b$ is $\mathcal W$ itself and the map 
 \begin{displaymath}
  \mathbf{L}i^*\mathcal E|_{X_0} \to \mathcal W|_{X_0}
 \end{displaymath}
 induces an isomorphism on the weight $b$ portion of the decomposition. Now we define $\mathcal E'$ to be the cone over the map $\mathcal E \to i_* \mathcal W$. Restricting the exact triangle to $X_0$ and remembering that computing weight spaces is an exact functor, we see that the weights of $\mathcal E'$ along $\mathcal Z$ are concentrated in $[a,b-1]$. We may also raise the weights by conjugating the previous procedure by dualization. Combining the two procedures, we can move the weights of any complex into $I$ up to $\mathcal Z$-torsion. 
\end{proof}

\begin{corollary} \label{corollary: window equivalence}
 Assume that the weights of the conormal sheaf of $\mathcal Z$ in $\mathcal X$ are all strictly negative and let $t_{\mathcal Z}$ be the weight of the relative canonical sheaf $\omega_{\mathcal Z \mid \mathcal X}$. Set $\mathcal U := \mathcal X \setminus \mathcal Z$ and $j: \mathcal U \to \mathcal X$ be the inclusion. Then, the functor
 \begin{displaymath}
  j^*: \weezer(I,\mathcal Z_0) \to \dbcoh{\mathcal U}
 \end{displaymath}
 is an equivalence whenever $I$ is interval of length $-t_{\mathcal Z} - 1$.
\end{corollary}

\begin{proof}
 This is an immediate consequence of Lemmas~\ref{lemma: fully-faithful} and \ref{lemma: essentially surjective}. 
\end{proof}

\begin{definition} \label{definition: complementary pieces}
 Let $s \in \Z$. Denote by $\mathcal C_s(\mathcal Z)$ the full subcategory of $\dbcoh{\mathcal Z_0}$ consisting of objects with weight $s$. 
\end{definition}

\begin{lemma} \label{lemma: complementary pieces pull back ok}
 Assume that the weights of the conormal sheaf of $\mathcal Z$ in $\mathcal X$ are all strictly negative. The functor 
 \begin{align*}
  \Upsilon_s : \mathcal C_s(\mathcal Z) & \to \dbcoh{\mathcal X} \\
  \mathcal E & \mapsto i_* \pi^* \mathcal E
 \end{align*}
 is fully-faithful. 
\end{lemma}

\begin{proof}
 We use the standard adjunctions. We have
 \begin{displaymath}
  \op{Hom}_{\mathcal X}( \Upsilon_s \mathcal E, \Upsilon_s \mathcal F) \cong \op{Hom}_{\mathcal Z}( \mathbf{L}i^*i_* \pi^* \mathcal E, \pi^* \mathcal F).
 \end{displaymath}
 The cohomology sheaves of $\mathbf{L}i^*i_* \pi^* \mathcal E$ are isomorphic to $\pi^* \mathcal E \otimes \op{Sym}^*(\mathcal T_{\mathcal Z \mid \mathcal X})$. Thus, there is a natural map 
 \begin{equation} \label{equation: a map in complementary pieces proof}
  \op{Hom}_{\mathcal Z}( \pi^* \mathcal E, \pi^* \mathcal F) \to \op{Hom}_{\mathcal X}( \Upsilon_s \mathcal E, \Upsilon_s \mathcal F).
 \end{equation}
 We first check that this is isomorphism. The cone of $\mathbf{L}i^*i_* \pi^* \mathcal E \to \pi^* \mathcal E$ has cohomology sheaves $\pi^* \mathcal E \otimes \op{Sym}^{\geq 1}(\Omega_{\mathcal Z \mid \mathcal X})$ which have weights $< s$. Let us show that
 \begin{displaymath}
  \op{Hom}_{\mathcal Z}(\pi^* \mathcal E \otimes \op{Sym}^{\geq 1}(\Omega_{\mathcal Z \mid \mathcal X}), \pi^* \mathcal F[s]) = 0
 \end{displaymath}
 for any $s$. This vanishing combined with a spectral sequence argument gives that the map in Equation~\eqref{equation: a map in complementary pieces proof} is an isomorphism. We have
 \begin{align*}
  \op{Hom}_{\mathcal Z}(\pi^* \mathcal E \otimes \op{Sym}^{\geq 1}(\Omega_{\mathcal Z \mid \mathcal X}), \pi^* \mathcal F[s]) & \cong \op{Hom}_{\mathcal Z_0}(\mathcal E \otimes \op{Sym}^{\geq 1}(\Omega_{\mathcal Z \mid \mathcal X}), \mathcal F \otimes \op{Sym}^*(\Omega_{\mathcal Z_0 \mid \mathcal Z})[s]) \\
  & \cong \op{Hom}_{\mathcal Z_0}(\mathcal E, \mathcal F \otimes \op{Sym}^*(\Omega_{\mathcal Z_0 \mid \mathcal Z}) \otimes \op{Sym}^{\geq 1}(\mathcal T_{\mathcal Z \mid \mathcal X})[s]).
 \end{align*}
 We can again factor through pushforward to the rigidification $\mathcal Z_0^{\mathbb{G}_m}$. In this case, the above Hom-space is zero as the weights of the right hand side are concentrated in $(s,\infty)$. 
 
 Next we have
 \begin{align*}
  \op{Hom}_{\mathcal Z}( \pi^* \mathcal E, \pi^* \mathcal F) & \cong \op{Hom}_{\mathcal Z_0}( \mathcal E, \mathcal F \otimes \op{Sym}^*(\Omega_{\mathcal Z_0 \mid \mathcal Z})).
 \end{align*}
 The only piece of $\mathcal F \otimes \op{Sym}^*(\Omega_{\mathcal Z_0 \mid \mathcal Z})$ in weight $s$ is $\mathcal F$. Thus, 
 \begin{displaymath}
  \op{Hom}_{\mathcal Z_0}( \mathcal E, \mathcal F \otimes \op{Sym}^*(\Omega_{\mathcal Z_0 \mid \mathcal Z})) \cong \op{Hom}_{\mathcal Z_0}( \mathcal E, \mathcal F)
 \end{displaymath}
 which gives fully-faithfulness.
\end{proof}

\begin{lemma} \label{lemma: semi-orthogonal decomposition of bigger windows}
 Assume that the weights of the conormal sheaf of $\mathcal Z$ in $\mathcal X$ are all strictly negative and let $t_{\mathcal Z}$ be the weight of the relative canonical sheaf $\omega_{\mathcal Z \mid \mathcal X}$. Assume that $v-u \geq -t_{\mathcal Z}$. There is a semi-orthogonal decomposition
 \begin{displaymath}
  \weezer([u,v],\mathcal Z_0) = \langle \Upsilon_{v}, \weezer([u,v-1],\mathcal Z_0) \rangle.
 \end{displaymath}
\end{lemma}

\begin{proof}
 Take $\mathcal E \in \weezer([u,v],\mathcal Z_0)$ and consider the exact triangle from the proof of Lemma~\ref{lemma: essentially surjective}
 \begin{displaymath}
  \mathcal E' \to \mathcal E \to i_* \mathcal W. 
 \end{displaymath}
 It is clear that from the definition that $i_* \mathcal W$ lies in the image of $\Upsilon_v$ and we saw already that $\mathcal E'$ lies in $\weezer([u,v-1],\mathcal Z_0)$. Thus, the two subcategories generate $\weezer([u,v],\mathcal Z_0)$. It remains to check semi-orthogonality. This follows as in the proof of Lemma~\ref{lemma: fully-faithful}. 
\end{proof}

\begin{remark} \label{remark: singular case}
 The reader should observe that everything goes through if $\mathcal X$ is singular but smooth along $\mathcal Z$. 
\end{remark}

\begin{definition} \label{definition: elementary strata and elementary wall crossing}
 Let $\mathcal X$ be a smooth algebraic stack of finite-type over $k$. A stacky BB stratum $\mathcal Z$ in $\mathcal X$ is called an \textbf{elementary stratum} if the weights of $\Omega_{\mathcal Z \mid \mathcal X}$ along $\mathcal Z$ are strictly negative.  
 
 A pair of elementary strata $\mathcal Z_-$ and $\mathcal Z_+$ is called an \textbf{elementary wall crossing} if $\mathcal Z_{-,0} = \mathcal Z_{+,0}$ and the two embeddings of $\mathbb{G}_m$ into the automorphisms of $\mathcal Z_{\pm,0}$ differ by inversion. 
\end{definition}

\begin{theorem} \label{theorem: elementary wall crossing}
 Assume we have elementary wall crossing, $\mathcal Z_-,\mathcal Z_+$. Fix $d \in \Z$. 
 \begin{enumerate}
  \item If $t_{\mathcal Z_+} < t_{\mathcal Z_-}$, then there are fully-faithful functors,
  \begin{displaymath}
   \Phi^+_d: \dbcoh{\mathcal U_-} \to \dbcoh{\mathcal U_+},
  \end{displaymath}
  and, for $-t_{\mathcal Z_-} + d \leq j \leq -t_{\mathcal Z_+} + d - 1$,
  \begin{displaymath}
   \widetilde{\Upsilon}_j^-: \mathcal C_j(\mathcal Z_-) \to \dbcoh{\mathcal U_+},
  \end{displaymath}
  and a semi-orthogonal decomposition,
  \begin{displaymath}
   \dbcoh{\mathcal U_+} = \langle \widetilde{\Upsilon}^-_{-t_{\mathcal Z_-}+d}, \ldots, \widetilde{\Upsilon}^-_{-t_{\mathcal Z_+}+d-1}, \Phi^+_d \rangle.
  \end{displaymath}
  \item If $t_{\mathcal Z_+} = t_{\mathcal Z_-}$, then there is an exact equivalence,
  \begin{displaymath}
   \Phi^+_d: \dbcoh{\mathcal U_-} \to \dbcoh{\mathcal U_+}.
  \end{displaymath}
  \item If $t_{\mathcal Z_+} > t_{\mathcal Z_-}$, then there are fully-faithful functors,
  \begin{displaymath}
   \Phi^-_d: \dbcoh{\mathcal U_+} \to \dbcoh{\mathcal U_-},
  \end{displaymath}
  and, for $-t_{\mathcal Z_+} + d \leq j \leq -t_{\mathcal Z_-} + d -1$,
  \begin{displaymath}
   \widetilde{\Upsilon}_j^+: \mathcal C_j(\mathcal Z_+) \to \dbcoh{\mathcal U_-},
  \end{displaymath}
  and a semi-orthogonal decomposition,
  \begin{displaymath}
   \dbcoh{\mathcal U_-} = \langle \widetilde{\Upsilon}^+_{-t_{\mathcal Z_+}+d}, \ldots, \widetilde{\Upsilon}^+_{-t_{\mathcal Z_-}+d-1}, \Phi^-_d \rangle.
  \end{displaymath}
 \end{enumerate}
\end{theorem}

\begin{proof}
 This is the same argument as for the proof of \cite[Theorem 3.5.2]{BFKGIT}. Again, we recall it in some detail. Swapping the roles of $\mathcal Z_+$ and $\mathcal Z_-$ we can assume that $t_{\mathcal Z_+} \leq t_{\mathcal Z_-}$. Choose intervals $I_- \subseteq I_+$ with the diameter of $I_{\pm}$ equal to $-t_{\mathcal Z_{\pm}} - 1$ and $d := \op{min} I_- = \op{min} I_+$. From Lemma~\ref{lemma: semi-orthogonal decomposition of bigger windows}, there is a semi-orthogonal decomposition
 \begin{displaymath}
  \weezer(I_+,\mathcal Z_0) = \langle \Upsilon^-_{-t_{\mathcal Z_-}+d}, \ldots, \Upsilon^-_{-t_{\mathcal Z_+}+d-1}, \weezer(I_-,\mathcal Z_0) \rangle
 \end{displaymath}
 Using Corollary~\ref{corollary: window equivalence}, we can pull back to $\mathcal U_+$ to get 
 \begin{displaymath}
  \dbcoh{\mathcal U_+} = \langle i^*_+ \circ \Upsilon^-_{-t_{\mathcal Z_-}+d}, \ldots, i^*_+ \circ \Upsilon^-_{-t_{\mathcal Z_+}+d-1}, i^*_+\weezer(I_-,\mathcal Z_0) \rangle.
 \end{displaymath}
 Applying Corollary~\ref{corollary: window equivalence} again, we know that $i_-^*$ induces an equivalence between $\weezer(I_-,\mathcal Z_0)$ and $\dbcoh{\mathcal U_-}$. We set 
 \begin{displaymath}
  \widetilde{\Upsilon}_j^- := i^*_+  \circ \Upsilon^-_j
 \end{displaymath}
 and
 \begin{displaymath}
  \Phi^+_d := i_+^* \circ (i^*_-)^{-1}
 \end{displaymath}
 to finish.
\end{proof}

\section{Stable sheaves on rational surfaces}

In this section, we apply Theorem~\ref{theorem: elementary wall crossing} using the well-known structure of semi-stable rank two torsion-free sheaves on rational surfaces \cite{EG,FQ,MW}.

Let $S$ be a smooth complex projective surface. In this section, we will show how to apply Theorem~\ref{theorem: elementary wall crossing} to wall-crossing of Gieseker stable sheaves obtained by variation of the polarization on $S$. 

\subsection{Basics}

Let us recall the main notions of stability. Let $L$ be an ample line bundle on $S$.

\begin{definition} \label{definition: stability}
 Let $E$ be coherent sheaf on $S$. The sheaf, $E$, is \textbf{Gieseker $L$-semi-stable} \cite{Gieseker} if it is torsion-free and for any proper subsheaf $F \subsetneq  E$ one has $\overline{p}_L(F) \leq \overline{p}_L(E)$ where $\overline{p}_L$ is the reduced Hilbert polynomial associated to the embedding given by $L$. If the inequality is strict for any proper subsheaf, $E$ is \textbf{Gieseker $L$-stable}. If $E$ is not Gieseker $L$-semi-stable, then $E$ is \textbf{Gieseker $L$-unstable}. 
 
 The sheaf, $E$, is \textbf{Mumford $L$-semi-stable} \cite{Mum62,Takemoto} if it is torsion-free and for any proper subsheaf $F \subset E$ one has $\mu_L(F) \leq \mu_L(E)$ where $\mu_L$ is the $L$-slope of the sheaf. Again if the inequality is always strict, $E$ is \textbf{Mumford $L$-stable} and if $E$ is not Mumford $L$-semi-stable then it is called \textbf{Mumford $L$-unstable}. 
\end{definition}

Fix invariants $c_0,c_1,c_2$ and consider the moduli functors, $\widetilde{\mathcal M}_L(c_0,c_1,c_2)$, $\widetilde{\mathcal Mum}_L(c_0,c_1,c_2)$, given by
\begin{align*}
   X & \mapsto \{\text{iso. classes of Gieseker $L$-s.s. families with } c_0(\mathcal F_x) = c_0, c_1(\mathcal F_x) = c_1, c_2(\mathcal F_x) = c_2 \} \\
   X & \mapsto \{\text{iso. classes of Mumford $L$-s.s. families with } c_0(\mathcal F_x) = c_0, c_1(\mathcal F_x) = c_1, c_2(\mathcal F_x) = c_2 \}.
\end{align*}

The following is well-known.

\begin{lemma} \label{lemma: moduli stack is finite-type}
 The functor, $\widetilde{\mathcal M}_L(c_0,c_1,c_2)$, is an algebraic stack of finite-type over $k$. The same is true for $\widetilde{\mathcal Mum}_L(c_0,c_1,c_2)$.
\end{lemma}

\begin{proof}
 First, $\widetilde{\mathcal M}_L(c_0,c_1,c_2)$ is open, \cite[Proposition 2.3.1]{HL}, in the stack of coherent sheaves on $S$ which is algebraic, \cite[Theorem 75.5.12]{stacks-project}. Thus, $\widetilde{\mathcal M}_L(c_0,c_1,c_2)$ is algebaic stack. There is a bounded family of $L$-semi-stable sheaves with fixed numerical invariants, \cite[Theorem 3.3.7]{HL}. Base changing the induced map to $\widetilde{\mathcal M}_L(c_0,c_1,c_2)$ gives a smooth, surjective map with source of finite-type. A similar argument shows the statement for $\widetilde{\mathcal Mum}_L(c_0,c_1,c_2)$. 
\end{proof}

Next, we address smoothness. Recall that $\mathcal E$ is Gieseker (Mumford) polystable if it is the direct sum of Gieseker (Mumford) stable sheaves.

\begin{lemma} \label{lemma: smoothness at wall points}
 Assume $K_S < 0$. Let $E$ be Gieseker $L$-polystable or simple. Then, $E$ is a smooth point of $\widetilde{\mathcal M}_L(c_0,c_1,c_2)$. A similar statement holds for $\widetilde{\mathcal Mum}_L(c_0,c_1,c_2)$. 
\end{lemma}

\begin{proof}
 Write $E = \bigoplus_{i \in I} E_i$ with each $E_i$ stable. First, since each $E_i$ is stable, or since $E$ is simple, and $S$ is proper, we have
 \begin{displaymath}
  \op{Hom}(E,E) \subset \op{M}_m(\Gamma(S,\mathcal O_S)) = M_m(k). 
 \end{displaymath}
 where $M_m(k)$ is $m \times m$-matrices in $k$.
 
 To check smoothness, it suffices to show that $\op{ext}^1(E, E) - \op{hom}(E, E) = \chi(\op{Ext}^*(E,E))$, which remains constant over $\mathcal M_L(c_0,c_1,c_2)$. It suffices, therefore, to show that $\op{Ext}^2(E, E) = 0$. From Serre duality, we have 
 \begin{displaymath}
  \op{ext}^2(E, E) = \op{h}^0( \mathcal Hom(E, E) \otimes \omega_S).
 \end{displaymath}
 Since $K_S < 0$, there is an inclusion $\omega_S \to \mathcal O_S$ and an inclusion
 \begin{displaymath}
  \mathcal Hom(E, E) \otimes \omega_S \to \mathcal Hom(E,E).
 \end{displaymath}
 Taking global sections, we see that $\op{h}^0( \mathcal Hom(E, E) \otimes \omega_S)$ counts the dimension of the subspace of global sections of $\op{Hom}(E,E)$ which vanish along $-K_S$. Since all global sections are constant over $S$, we have 
 \begin{displaymath}
  \op{h}^0( \mathcal Hom(E, E) \otimes \omega_S) = 0.
 \end{displaymath} 
 Since Mumford stability implies Gieseker stability, we get the same statement for Mumford semi-stable sheaves. 
\end{proof}

\subsection{Rank two stable sheaves}

We now restrict ourselves to the case $S$ is a rational surface and $c_0 = 2$. We recall the results of Friedman and Qin \cite{FQ}, see also \cite{EG,MW}. Let $L_+$ and $L_-$ be two ample line bundles on $S$. For a divisor $\xi$ satisfying $\xi \equiv c_1 (\op{mod} 2)$ and $c_1^2-4c_2 \leq \xi^2 \leq 0$, consider the hyperplane in $\op{Amp}(S)_{\mathbb{R}}$ given by 
\begin{displaymath}
 W^{\xi} := \{ D \mid D \cdot \xi = 0 \}.
\end{displaymath}
The hyperplane $W^{\xi}$ is called the wall associated to $\xi$. For simplicity, we shall assume that the line segment joining $L_+$ and $L_-$ intersects only a single $W^{\xi}$, determined by a unique $\xi$. The more general case, where rational multiples of $\xi$ may remain integral and define the same wall, requires minimal modification of the argument \cite{FQ}. Denote the polarization given by the intersection of the line and the hyperplane by $L_0$. We assume that $L_0$ lies in no other walls. 

We have two inclusions 
\begin{displaymath}
 \widetilde{\mathcal M}_{L_-}(c_1,c_2) \subseteq \widetilde{\mathcal Mum}_{L_0}(c_1,c_2) \supseteq \widetilde{\mathcal M}_{L_+}(c_1,c_2).
\end{displaymath}
where we change notation
\begin{displaymath}
 \widetilde{\mathcal M}_L(2,c_1,c_2) =: \widetilde{\mathcal M}_L(c_1,c_2)
\end{displaymath}
to reflect that the focus of our attention is upon rank two sheaves. Note the switch to Mumford semi-stable in the wall. Friedman and Qin study Mumford $L_0$-semi-stable sheaves of a particular form. 

\begin{definition}
 Let $Z^k(F)$ be the set  of sheaves $E$ occuring in a short exact sequence
 \begin{displaymath}
  0 \to I_{Z_1}(F) \to E \to I_{Z_2}(\Delta - F) \to 0
 \end{displaymath}
 where $c_1(\Delta) = c_1$, $\xi = 2F - \Delta$, $Z_1 \in \op{Hilb}^k(S)$, and $Z_2 \in \op{Hilb}^{l_{\xi}-k}(S)$ with 
 \begin{displaymath}
  l_{\xi} = (4c_2 - c_1^2 + \xi^2)/4.
 \end{displaymath}
 
 Denote the associated substack of $\widetilde{\mathcal Mum}_{L_0}(c_1,c_2)$ by $\widetilde{\mathcal Z}^k(F)$. 
\end{definition}

\begin{proposition}
 The substack $\widetilde{\mathcal Z}^k(F)$ is closed in $\widetilde{\mathcal Mum}_{L_0}(c_1,c_2)$ and $\widetilde{\mathcal Mum}_{L_0}(c_1,c_2)$ is smooth along $\widetilde{\mathcal Z}^k(F)$. 
\end{proposition}

\begin{proof}
 A Mumford $L_0$-semi-stable sheaf $E$ lies in $Z^{k}(F)$ if and only if there is a surjective map
 \begin{displaymath}
  E(-F) \to I_{Z_2} \to 0
 \end{displaymath}
 with $\xi = 2F - \Delta$ and $l(Z_2) = l_{\xi} - k$. This is a closed condition as it states $E(-F)$ lies in the Quot scheme for the Hilbert polynomial associated to $I_{Z_2}$. 
 
 By \cite[Lemma 2.2]{FQ}, the non-split extensions in $Z^k(F)$ are simple. The split extensions are polystable. So Lemma~\ref{lemma: smoothness at wall points} gives the last statement.
\end{proof}

Applying this proposition with the switch $F \to \Delta -F$ also gives the corresponding statement for $\widetilde{\mathcal Z}^k(\Delta - F)$. 

The stacks $\widetilde{\mathcal Z}^k(F)$ admit particularly simple geometric descriptions. Let $\mathcal E^k(F)$ be the coherent sheaf on $H^k := \op{Hilb}^k(S) \times \op{Hilb}^{l_{\xi} - k}(S)$ classifying the extensions appearing in $Z^k(F)$. By \cite[Lemma 2.6]{FQ}, $\mathcal E^k(F)$ is locally-free. Let
\begin{displaymath}
 X^k(F) := \underline{\op{Spec}}(\mathcal E^k(F))
\end{displaymath}
be the associated geometric vector bundle on $H^k$. There is a natural action of $\mathbb{G}_m^2$ on $X^k(F)$ given by endomorphisms of $I_{Z_1}(F) \oplus I_{Z_2}(\Delta - F)$. 

\begin{proposition}
 There are isomorphisms
 \begin{align*}
  \widetilde{\mathcal Z}^k(F) & \cong [X^k(F)/\mathbb{G}_m^2] \\
  \widetilde{\mathcal Z}^k(\Delta-F) & \cong [X^k(\Delta-F)/\mathbb{G}_m^2].
 \end{align*}
\end{proposition}

\begin{proof}
 From \cite[Lemma 2.2.i]{FQ}, given $E$ in $\widetilde{\mathcal Z}^k(F)$, then $Z_1$ and $Z_2$ are uniquely determined and the map $\mathcal I_{Z_1}(F) \to E$ is unique up to scaling. So the extension class for a given $E$ is determined up to scaling. 
\end{proof}

To move closer to schemes, we now rigidify our stacks and remove the residual $\mathbb{G}_m$ coming from multiples of the identity. We denote the rigidified stacks by removing the tilde, e.g. the $\mathbb{G}_m$-rigidified moduli stack of Mumford $L_0$-semi-stable sheaves will be denoted by $\mathcal Mum_{L_0}(c_1,c_2)$. We do this now, at the current point in the argument, to guarantee the following result holds.

\begin{proposition}
 The substacks $\mathcal Z^k(F)$ and $\mathcal Z^{l_{\xi} - k}(\Delta - F)$ form an elementary wall crossing in the stack $\mathcal Mum_{L_0}(c_1,c_2)$. 
\end{proposition}

\begin{proof}
 We have a presentation of $\mathcal Z^k(F)$ given by
 \begin{displaymath}
  \mathbb{G}_m \times X^k(F) = \mathbb{G}_m^2/\mathbb{G}_m \times X^k(F) \rightrightarrows X^k(F)
 \end{displaymath}
 Now take the $\mathbb{G}_m$ given by the first summand in $\mathbb{G}_m$, i.e scalar endomorphisms of $I_{Z_1}(F)$. Under this $\mathbb{G}_m$-action, $X^k(F)$ contracts onto the zero locus, $H^k$, so is a BB-stratum. This also gives the morphisms $l: \mathbb{G}_m \times X^k(F) \to \mathbb{G}_m^2 \times X^k(F)$ and $l: \mathbb{G}_m \times X^k(F) \to \mathbb{G}_m^2/\mathbb{G}_m \times X^k(F)$. The adjoint action is trivial on the first factor and is the action on the second. Thus, it extends to $\mathbb{A}^1$ and $\mathcal Z^k(F)$ is a stacky BB stratum. Computations in \cite[Section 3]{FQ} identify the conormal sheaf of $\mathcal Z^k(F)$ with $\mathcal E^{l_{\xi}-k}(\Delta - F)$ which has weight $-1$ with respect to this $\mathbb{G}_m$-action. So $\mathcal Z^k(F)$ is an elementary stratum. Similarly, one shows that $\mathcal Z^{l_{\xi} -k}(\Delta - F)$ is an elementary stratum. Since we have rigidified, the two choices of $\mathbb{G}_m$ actions on the fixed substacks differ by inversion.
\end{proof}

From here on, we assume that 
\begin{displaymath}
 L_- \cdot (2F - \Delta) < 0 < L_+ \cdot (2F - \Delta).
\end{displaymath}
and 
\begin{displaymath}
 \omega_S^{-1} \cdot (2F - \Delta) \geq 0. 
\end{displaymath}

Next, we want to compare the moduli stacks $\mathcal M_{L_+}(c_1,c_2)$ and $\mathcal M_{L_-}(c_1,c_2)$ via a sequence to elementary wall-crossings. 

To do so, we consider the following intermediate stacks 
\begin{displaymath}
 \mathcal M^{\leq l, \geq t} := \mathcal Mum_{L_0}(c_1,c_2) \setminus \left( \bigcup_{k > l} \mathcal Z^k(F) \cup \bigcup_{k < t} \mathcal Z^k(\Delta - F) \right). 
\end{displaymath}

\begin{lemma} \label{lemma: ends of wall crossings}
 We have 
 \begin{align*}
  \mathcal M_{L_+}(c_1,c_2) & = \mathcal M^{\leq l_{\xi}, \geq l_{\xi}+1} \\
  \mathcal M_{L_-}(c_1,c_2) & = \mathcal M^{\leq -1, \geq 0}.
 \end{align*}
\end{lemma}

\begin{proof}
 This is \cite[Lemma 3.2.ii]{FQ}.
\end{proof}

We get a sequence of elementary wall crossing given by
\begin{displaymath}
 \mathcal M^{\leq l, \geq l+1} \subset \mathcal M^{\leq l+1, \geq l+1} \supset \mathcal M^{\leq l+1, \geq l+2}.
\end{displaymath}
Applying Theorem~\ref{theorem: elementary wall crossing}, we get the following statement. 

\begin{proposition} \label{proposition: one wall crossing}
 With the assumptions as above, there is a semi-orthogonal decomposition
 \begin{displaymath}
  \dbcoh{ \mathcal M^{\leq l+1, \geq l+2} } = \left\langle \underbrace{\dbcoh{H^l}, \ldots, \dbcoh{H^l}}_{\mu_{\xi}}, \dbcoh{ \mathcal M^{\leq l, \geq l+1}}\right\rangle 
 \end{displaymath}
 where 
 \begin{align*}
  H^l & := \op{Hilb}^l(S) \times \op{Hilb}^{l_{\xi} - l}(S) \\
  \mu_{\xi} & := \omega_S^{-1} \cdot (2F - \Delta) = \omega_S^{-1} \cdot \xi.
 \end{align*}
\end{proposition}

\begin{proof}
 This is an immediate application of Theorem~\ref{theorem: elementary wall crossing} using \cite[Lemma 2.6]{FQ} to compute the number of copies of $\dbcoh{H^l}$. 
\end{proof}

\begin{corollary} \label{corollary: SOD for whole wall crossing}
 Let $S$ be a smooth rational surface over $\C$ with $L_-$ and $L_+$ ample lines bundles on $S$ separated by a single wall defined by unique divisor $\xi$ satisfying
 \begin{gather*}
  L_- \cdot \xi < 0 < L_+ \cdot \xi \\
  0 \leq \omega_S^{-1} \cdot \xi.
 \end{gather*}
 Let $\mathcal M_{L_{\pm}}(c_1,c_2)$ be the $\mathbb{G}_m$-rigidified moduli stack of Gieseker $L_{\pm}$-semi-stable torsion-free sheaves of rank $2$ with first Chern class $c_1$ and second Chern class $c_2$.
 
 There is a semi-orthogonal decomposition 
 \begin{gather*}
  \dbcoh{ \mathcal M_{L_+}(c_1,c_2) } = \left\langle \underbrace{\dbcoh{H^{l_{\xi}}}, \ldots,  \dbcoh{H^{l_{\xi}}}}_{\mu_{\xi}}, \ldots  \right. \\
  \left. \underbrace{\dbcoh{H^0}, \ldots, \dbcoh{H^0}}_{\mu_{\xi}}, \dbcoh{ \mathcal M_{L_-}(c_1,c_2)} \right\rangle 
 \end{gather*}
 where 
 \begin{align*}
  l_{\xi} & := (4c_2 - c_1^2 + \xi^2)/4 \\
  H^l & := \op{Hilb}^l(S) \times \op{Hilb}^{l_{\xi} - l}(S) \\
  \mu_{\xi} & := \omega_S^{-1} \cdot \xi
 \end{align*}
 with the convention that $\op{Hilb}^0(S) := \op{Spec} \C$.
\end{corollary}

\begin{proof}
 This is an iterated application of Proposition~\ref{proposition: one wall crossing} using Lemma~\ref{lemma: ends of wall crossings} to identify first and last moduli spaces. 
\end{proof}

\begin{remark}
 The results on wall crossing of moduli spaces of semi-stable sheaves in \cite{EG,FQ,MW} were originally obtained to compute the change in the Donaldson invariants under change of the polarization. Thus, Corollary~\ref{corollary: SOD for whole wall crossing} can be viewed as a categorification of that wall-crossing formula. 
\end{remark}

\subsection{An example}

Let us work out the consequences of Corollary~\ref{corollary: SOD for whole wall crossing} in the case $S = \mathbb{P}^1 \times \mathbb{P}^1$ with $c_1 = 5H_1 + 5H_2$ and $c_2 = 14$. We get the following decomposition of the ample cone into walls and chambers
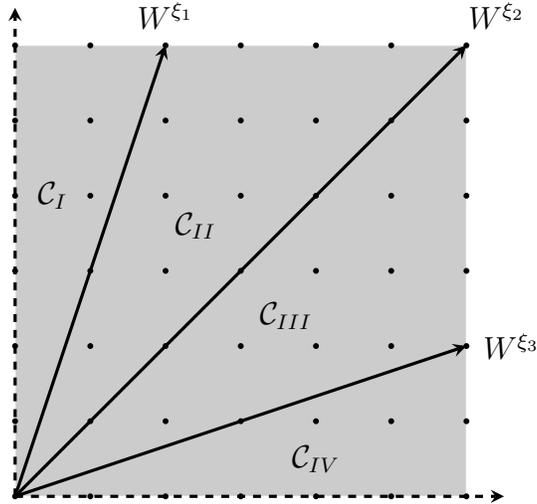
\begin{center}
\begin{tikzpicture}
  [scale=1, vertex/.style={circle,draw=black!100,fill=black!100,thick, inner sep=0.5pt,minimum size=0.5mm}, cone/.style={->,very thick,>=stealth}, wall/.style={->,dashed,very thick,>=stealth}]
  \filldraw[fill=black!20!white,draw=white!100]
    (0,0) -- (6,0) -- (6,6) -- (0,6) -- (0,0);
  \draw[cone] (0,0) -- (6,6);
  \draw[cone] (0,0) -- (6,2);
  \draw[cone] (0,0) -- (2,6);
  \draw[wall] (0,0) -- (0,6.5);
  \draw[wall] (0,0) -- (6.5,0);
  \node at (6.4,6.4) {$W^{\xi_2}$};
  \node at (2,6.4) {$W^{\xi_1}$};
  \node at (6.6,2) {$W^{\xi_3}$};
  \node at (0.5,4) {$\mathcal C_I$};
  \node at (2.4,3.6) {$\mathcal C_{II}$};
  \node at (4,0.5) {$\mathcal C_{IV}$};
  \node at (3.6,2.4) {$\mathcal C_{III}$};
 \foreach \x in {0,1,...,6}
  \foreach \y in {0,1,...,6}
  {
    \node[vertex] at (\x,\y) {};
  }
\end{tikzpicture}
\end{center}
where 
\begin{align*}
 \xi_1 & = 3H_1-H_2 \\
 \xi_2 & = H_1 - H_2 \\
 \xi_3 & = -H_1 + 3H_2.
\end{align*}
We have
\begin{align*}
 l_{\xi_1} & = 0 , \mu_{\xi_1}  = 5 \\
 l_{\xi_2} & = 1 , \mu_{\xi_2}  = 0 \\
 l_{\xi_3} & = 0 , \mu_{\xi_3}  = 5.
\end{align*}

Let $\mathcal M_J$ denote the (rigidified) moduli stack of Gieseker $L$-semi-stable torsion free sheaves of rank $2$ with $c_1 = 5H_1+5H_2$ and $c_2 = 14$, $L \in \mathcal C_J$ and $J \in \{I,II,III,IV\}$. Applying Corollary~\ref{corollary: SOD for whole wall crossing} to crossing of $W_1$, we see that there is a semi-orthogonal decomposition
\begin{displaymath}
 \dbcoh{\mathcal M_{II}} = \left\langle E_1,\ldots, E_5, \dbcoh{\mathcal M_I} \right\rangle
\end{displaymath}
with $E_i$ exceptional. Applying it to $W_2$, we get an equivalence
\begin{displaymath}
 \Phi: \dbcoh{\mathcal M_{II}} \overset{\sim}{\to} \dbcoh{\mathcal M_{III}}.
\end{displaymath}
Applying it to $W_3$, we get a semi-orthogonal decomposition
\begin{displaymath}
 \dbcoh{\mathcal M_{III}} = \left\langle F_1,\ldots, F_5, \dbcoh{\mathcal M_{IV}} \right\rangle
\end{displaymath}
with $F_i$ exceptional.

The involution, $i$, that exchanges the two factors of $\mathbb{P}^1 \times \mathbb{P}^1$ induces isomorphisms
\begin{displaymath}
 \mathcal M_I \cong \mathcal M_{IV}
\end{displaymath}
and
\begin{displaymath}
 \mathcal M_{II} \cong \mathcal M_{III}
\end{displaymath}
so there are really only two moduli spaces here. Note however, that the equivalence $\Phi$ is not the pullback $i^*$ as it leaves unchanged sheaves that semi-stable in both chambers. Combining the two equivalences, we get an interesting autoequivalence of $i^* \circ \Phi$ of $\dbcoh{\mathcal M_{II}}$.

%%%%%%%%%%%%%%%%%%%%%%%%%%%%%%%%%%%%%%%%%%%%%%%%%%%%%%%%%%%%%%%%%%%%%%%%%%%%

\end{document}